\documentclass[a4paper, 12pt, twoside, reqno]{amsart}

\makeatletter
\newcommand{\figcaption}[1]{\def\@captype{figure}\caption{#1}}
\newcommand{\tblcaption}[1]{\def\@captype{table}\caption{#1}}
\makeatother
\usepackage{caption}
\usepackage{mymacros, oitp}
\title[]{The web of reflexive polygons is connected}

\author[M.~Miura]{Makoto Miura}
\address{Research Institute for Mathematical Sciences, Kyoto University, Kyoto 606-8502, Japan}
\email{miurror.jp@gmail.com}

\begin{document}
\begin{abstract}
  We discuss the problem on the connectedness of various webs of lattice polytopes by 
  introducing a geometric point of view from the toric Mori theory.
  To this end,
  we provide a combinatorial description of toric Sarkisov links 
  in terms of certain sets of lattice points, which we call primitive generating sets.
  In two dimensions, the description is further translated into the language of lattice polygons.
  As an application, we prove in two ways (constructive and non-constructive) that reflexive 
  or terminal polygons form a single connected web via inclusion relations
  even without taking modulo unimodular equivalences.
\end{abstract}
\maketitle
\section{Introduction}
As a byproduct of their famous works on the classification of reflexive polytopes,
Kreuzer and Skarke have shown that
the web of $d$-dimensional reflexive polytopes is connected modulo unimodular equivalences if $d\le 4$
\cite{MR1901220}.
This is important in relation to the unsolved mathematical problem, so-called \emph{Reid's fantasy} \cite{MR909231},
which asks whether the web of smooth Calabi--Yau 3-folds is connected via geometric transitions.
In fact,
their result
implies that the web of 
smooth Calabi--Yau 3-folds
described as anticanonical hypersurfaces
in toric varieties (in the sense of Batyrev \cite{MR1269718} and 
Fredrickson \cite{fredrickson2015generalized})
is connected via geometric transitions and flops. 
Since their argument is purely combinatorial and relies on the results of computer-aided classification
of reflexive polytopes,
it is still an interesting problem to analyze
the web of reflexive polytopes from a more geometric point of view.

The present paper provides an argument based on the perspectives from the Mori theory in birational geometry. 
Let $Z$ be a smooth or mildly singular projective variety over the field of complex numbers $\bC$.
By running the \emph{minimal model program} (MMP, for short), 
one conjecturally obtains a sequence of elementary birational maps directed 
by the canonical divisor $K_Z$,
\begin{equation}
  \label{eq:mmp}
   \varphi: Z=X_0 \dashrightarrow X_1 \dashrightarrow \cdots \dashrightarrow X_m =X,
\end{equation}
where
$X$ is either a minimal model or a Mori fiber space $p: X\rightarrow S$.
Although the output $X$ is not unique, 
it is known that
different birational minimal models are related by a sequence of flops \cite{MR2426353}
and different birational Mori fiber spaces 
are related by a sequence
of Sarkisov links \cite{MR1311348}, \cite{MR3019454}.
Thus, in either case, one obtains a single web connecting all birational models via
elementary birational maps if the MMP works.
For projective toric varieties, 
the Mori theory has been established by the seminal paper of Reid
\cite{MR717617} and developed 
through a number of studies; for example,
\cite[Chapter 14]{MR1875410}, \cite{MR2097403}, \cite[Chapters 14--15]{MR2810322}.
By running MMP starting from a $\bQ$-factorial projective toric variety $Z$,
one always ends up with a toric Mori fiber space (we do not assume that it has only terminal singularities).
The resulting web of birational models 
can be coarse-grained into the web of Fano polytopes 
via inclusion relations.
Here a \emph{Fano polytope} (also known as a \emph{$\bQ$-Fano polytope}) is a lattice polytope containing the
origin as an interior point whose vertices are all primitive lattice points. 
This coarse-graining is done
by taking the convex hull $\Conv G(\Sigma)$ of 
the set of primitive ray generators $G(\Sigma)$ of each projective fan $\Sigma$.
For an elementary birational map $X_{\Sigma} \dashrightarrow X_{\Sigma'}$ in \pref{eq:mmp},
we have inclusions  $G(\Sigma)\supset G(\Sigma')$ and $\Conv G(\Sigma)\supset \Conv G(\Sigma')$.

The web of Fano polytopes is connected
since there is a sequence 
$\nabla \subset \Conv(\nabla\cup \nabla') \supset \nabla'$ for any pair of Fano polytopes $\nabla$ and $\nabla'$.
Hence, a natural problem is whether 
the web of a certain restricted class of Fano polytopes is connected.
If the answer to the following \pref{pb:connectedness} is yes,
we say that the class of lattice polytopes is \emph{globally connected}.
\begin{problem}
  \label{pb:connectedness}
  For a given class of lattice polytopes,
  do they form a
  single connected web via inclusion relations?
\end{problem}

We introduce five known classes of lattice polytopes including reflexive polytopes,
for which one could naturally ask \pref{pb:connectedness}.
Let $N\simeq \bZ^d$ and $M=\Hom(N,\bZ)$
be the dual pair of free abelian groups of rank $d$ and 
$N_\bR\coloneqq N\otimes_\bZ \bR$.
The \emph{polar dual} and the \emph{Mavlyutov dual} of a lattice polytope $\nabla$ in $N_\bR$ is defined as
\begin{equation}
  \nabla^* \coloneqq  \lc u\in M_\bR \relmid \la u, v \ra \ge -1 \text{ for all }  v\in \nabla \rc 
  \quad \text{ and } \quad [\nabla^*] \coloneqq \Conv (\nabla^* \cap M),
\end{equation}
respectively.
A reflexive (resp.\ pseudoreflexive) polytope is defined such as
the polar duality (resp.\ the Mavlyutov duality) holds among the same class of lattice polytopes.
Namely, $\nabla$ is \emph{reflexive} if $\nabla^*$ is also a lattice polytope,
and \emph{pseudoreflexive} if $[[\nabla^*]^*] = \nabla$ holds.
While these polytopes are well suited to mirror symmetry as shown by \cite{MR1269718} and \cite{mavlyutov2011mirror},
the following polytopes are more fitted to the arguments in the toric Mori theory.
A lattice polytope $\nabla$ is called \emph{canonical} 
if the origin is the unique interior lattice point, 
\emph{terminal} 
if $\nabla\cap N$ consists of 
the vertices of $\nabla$ and the origin,
and \emph{almost pseudoreflexive}
if the Mavlyutov dual $[\nabla^*]$ 
contains the origin as an interior point.
We shortly remark the geometric meaning of these three classes of lattice polytopes.
A canonical (resp.\ terminal) polytope $\nabla$ corresponds to 
a toric Fano variety with at worst canonical (resp.\ terminal) singularities \cite{MR717617},
which is the anticanonical model of 
the toric variety defined by a projective simplicial fan $\Sigma$ satisfying $\Conv G(\Sigma)=\nabla$. 
An almost pseudoreflexive polytope $\nabla$ corresponds to
a toric Fano variety whose general \emph{elephant} (i.e., 
a member of the anticanonical linear system $\abs{-K}$) is a Calabi--Yau variety.
In fact, the latter is equivalent to that
the Newton polytope $[\nabla^*]$ is
almost pseudoreflexive (and hence, pseudoreflexive in this case) as proved
by \cite[Theorem 2.23]{MR3858013} (and by \cite[Theorem 1]{MR3515305} for one direction). 
There are the following implications among these five classes of polytopes:
\begin{equation}
   \text{terminal} \ \Rightarrow \ \text{canonical} \ \Leftarrow \ \text{almost pseudoreflexive}
   \ \Leftarrow \ \text{pseudoreflexive} \ \Leftarrow \ \text{reflexive}.
\end{equation}
Only the second implication is nontrivial, which follows because $[\nabla^*]^*\supset \nabla$ contains only the origin as 
an internal lattice point
by \cite[Lemma 1.5]{MR3515305}.
A Fano polytope contained in an almost pseudoreflexive (resp.\ canonical, terminal) polytope
is again an almost pseudoreflexive (resp.\ canonical, terminal) polytope.
Hence, by running MMP, \pref{pb:connectedness} for these three classes of 
Fano polytopes comes down to \pref{pb:sarkisov} below.
\begin{problem}
  \label{pb:sarkisov}
For any pair of almost pseudoreflexive (resp.\ canonical, terminal) polytopes 
represented by toric Mori fiber spaces,
is there a pair of representatives 
connected by
a sequence of Sarkisov links 
consisting only of birational models whose associated polytopes are
almost pseudoreflexive (resp.\ canonical, terminal)?
\end{problem}

Note that the global connectedness of $d$-dimensional almost pseudoreflexive polytopes implies 
that of $d$-dimensional pseudoreflexive polytopes.
In fact, this is clear
by taking the pseudoreflexive hull $[[\nabla^*]^*]$ for 
each almost pseudoreflexive polytope $\nabla$
in a sequence 
connecting a pair of pseudoreflexive polytopes.
Together with the fact that a pseudoreflexive polytope is always reflexive when $d\le 4$ \cite{MR1397227}, 
\pref{pb:sarkisov} 
for almost pseudoreflexive polytopes possibly re-proves the connectedness of 
the web of four-dimensional reflexive polytopes due to \cite{MR1901220}
without taking modulo unimodular equivalences nor using the classification.

In order to make \pref{pb:sarkisov} more precise,
we explore the combinatorial description of 
Sarkisov links for toric Mori fiber spaces.
However, it turns out that there is an obstacle in the description of Sarkisov links only in terms of
Fano polytopes (\pref{eg:obstacle}). 
On the other hand, we obtain
a satisfactory description of Sarkisov links
in terms of primitive generating sets in arbitrary dimensions,
which are the combinatorial objects (corresponding to $G(\Sigma)$ above)
that appear between projective fans and Fano polytopes.
In \pref{sc:pgs}, we introduce and develop the notion of primitive generating sets.
In \pref{sc:sarkisov}, we prove \pref{th:sarkisovA}, which gives a combinatorial description of Sarkisov links 
in terms of primitive generating sets.

In two dimensions, things becomes very simple.
Especially, the description in \pref{th:sarkisovA} can be further translated into the 
language of Fano polygons without any additional conditions (\pref{pr:twodim}).
Using this simplicity, 
we give two proofs (constructive and non-constructive) for the following theorem in
\pref{sc:reflexive} by solving \pref{pb:sarkisov} in two dimensions:
\begin{theorem}
   \label{th:polygon}
   The web of reflexive polygons 
   are connected via inclusion relations
   without taking modulo unimodular equivalences.
   The same is true for the web of terminal polygons.
\end{theorem}
The latter proof also implies that 
elliptic elephants in smooth projective rational surfaces
can be connected 
via smooth transitions associated with blow-ups and blow-downs 
decomposing a given birational map between ambient spaces (\pref{cr:ellip}).
This could be regarded as a toy version of Reid's fantasy in one dimension.

Note that, for a lattice polygons, reflexive is equivalent to canonical,
so that \pref{th:polygon} yields a complete answer to \pref{pb:connectedness} 
for all the five classes of polytopes in two dimensions.
On the other hand,
the proof is still based on the classification of smooth toric Mori fiber surfaces 
and the hand-picked choice of sequences of Sarkisov links. 
Hopefully, more detailed study of the Sarkisov program could improve the argument 
and provide a way to solve \pref{pb:connectedness}
in higher dimensions and also a new perspective to Reid's
fantasy for Calabi--Yau 3-folds.\\


\emph{Acknowledgements.}
This paper is dedicated to Professor Shinobu Hosono on the occasion of his sixtieth birthday.
The author thanks Atsushi Ito for helpful discussions.
He was supported by Grants-in-Aid for Scientific Research (21K03156).

\section{primitive generating sets}
\label{sc:pgs}
\begin{definition}
We call a finite set $A\subset  N$ a \emph{primitive generating set} of $N_\bR$ if $A$ 
consists of primitive lattice points and generates $N_\bR$ as a cone, i.e.,
$N_\bR = \lc \sum_{v\in A}\lambda_v v \relmid \lambda_v \ge 0 \text{ for all } v\in A \rc$.
\end{definition}

\begin{definition}
   \label{df:reduction}
   An inclusion $A\supset A'$ of primitive generating sets
   is said to be a \emph{reduction} of $A$ 
   if $\abs{A}=\abs{A'}+1$. 
   We denote by $A\supsetdot A'$ a reduction of $A$.
\end{definition}

\begin{definition}
   \label{df:fiberstr}
   For a primitive generating set $A$,
   an inclusion $\Af \subset A$ is said to be
   a \emph{fiber structure} on $A$ if $\Af\ne \emptyset$ 
   is a primitive generating set of the linear span $L$ of $\Af$ and $\Af = L\cap A$.
   A fiber structure $\Af \subset A$ defines 
   the associated exact sequence
   \begin{equation}
      \label{eq:exact}
    \begin{tikzcd}
      0\arrow[r] & \Nf \arrow[r] & N \arrow[r,"\pi"] & \Nb \arrow[r] & 0,
    \end{tikzcd}
   \end{equation}
   where $\Nf \coloneqq N\cap L$ and $\Nb \coloneqq N/\Nf$. 
   For any $v\in A \setminus \Af$, the ray $\bR_+\pi(v)$ in $(\Nb)_\bR$ has the unique 
   primitive generator $\pibar(v) \in \Nb$.
   Thus there is a map 
   \begin{equation}
      \label{eq:pibar}
      \pibar: A \setminus \Af\rightarrow 
      \Ab \coloneqq \lc \pibar(v) \relmid v \in A\setminus \Af  \rc \subset \Nb.
   \end{equation}
   We call $\Ab$ the \emph{base} of the fiber structure $\Af \subset A$,
   which turns out to be a primitive generating set of $(\Nb)_\bR$.
   The map \pref{eq:pibar} is also denoted by $\pibar: A \dashrightarrow \Ab$. 

   A fiber structure $\Af\subset A$ is called \emph{irreducible} if $\abs{A}=\abs{\Af}+\abs{\Ab}$ holds.
   A \emph{Mori fiber structure} is an irreducible fiber structure
   satisfying $\abs{\Af} = \dim L+1$. 
   We call a primitive generating set $A$ that admits a Mori fiber structure
   a \emph{Mori fiber primitive generating set}.
   A Mori fiber structre $\Af \subset A$ on a Mori fiber primitive generating set $A$
   is denoted by $\Af \subsetm A$. 
\end{definition}
By abuse of notation,
we use all the notions in \pref{df:reduction} and \pref{df:fiberstr}
also for Fano polytopes
by replacing 
a primitive generating set $\nabla \cap \Nprim$
with
a Fano polytope $\nabla$,
where $\Nprim$ is the set of primitive lattice points in $N$.
Thus, for a reduction $\nabla\supsetdot \nabla'$ of Fano polytopes, 
$\nabla'$ is a Fano polytope defined as
the convex hull of all primitive lattice points in $\nabla\setminus \lc v\rc$
for a vertex $v\in \nabla$.
Similarly, for a Mori fiber structure $\nablaf \subsetm \nabla$ of Mori fiber Fano polytopes 
(or \emph{Mori fiber polytopes}, for short), 
$\nablaf$ is a terminal simplex and $\abs{\nabla\cap \Nprim}=\abs{\Ab} + \dim L +1$ holds,
where $\Ab$ is the base of $\nablaf\subsetm\nabla$. 
Notice that, even for a fiber structure $\nablaf\subset \nabla$ of Fano polytopes,
the base $\Ab$ needs not to be a Fano polytope.

Let $X_\Sigma$ be a $\bQ$-factorial projective toric variety defined by a simplicial projective fan $\Sigma$.
Denote by $\overline{\NE}(X_\Sigma)$ 
the closed convex cone generated by
numerical equivalence classes of curves.
The following 
is evident from the description of extremal contractions,
e.g., \cite[Proposition 15.4.5]{MR2810322}.
\begin{lemma}
   \label{lm:contractions}
   Let $\varphi_R: X_\Sigma \rightarrow X_{\Sigma'}$ be an extremal contraction of
   a $\bQ$-factorial projective toric variety $X_\Sigma$ with respect to 
   an extremal ray $R \subset \overline{\NE}(X_\Sigma)$. 
   If $\varphi_R$ is a small contraction, it preserves $G(\Sigma) = G(\Sigma')$.
   If $\varphi_R$ is a divisorial contraction, it gives a reduction $G(\Sigma) \supsetdot G(\Sigma')$.
   If $\varphi_R$ is a fibering contraction, 
   it gives a Mori fiber structure $G(\Sigmaf) \subsetm G(\Sigma)$ with the base $G(\Sigma')$
   for a naturally defined subfan $\Sigmaf$ of $\Sigma$.
\end{lemma}

\begin{remark}
   Note that, 
   if $\varphi_R$ is a fibering contraction (i.e., the last case in \pref{lm:contractions}), 
   $X_\Sigma$ is automatically a toric Mori fiber space
   with the fibers isomorphic to $X_{\Sigmaf}$, that is, $K_{X_{\Sigma}}.R<0$.
   In fact, $R$ is the outer normal of a facet of the pseudoeffective cone 
   $\overline{\Eff}(X_{\Sigma})$,
   and $-K_{X_\Sigma}$ is big so that its numerical class is contained in the interior of $\overline{\Eff}(X_{\Sigma})$.
   Since $\overline{\Eff}(X_\Sigma)$ is strongly convex, $K_{X_{\Sigma}}.R<0$ always holds.
\end{remark}
For any primitive generating set $A$, 
there exits a simplicial projective fan $\Sigma$ such that $G(\Sigma) =A$,
which we call a \emph{projective $A$-maximal fan},
following the case
where $A$ is (the set of primitive lattice points in) a reflexive polytope
\cite[Definition 2.1]{fredrickson2015generalized}.
The following is a key lemma which shows the converse of \pref{lm:contractions} also holds in a sense:
\begin{lemma}
  \label{lm:lift}
  For any primitive generating set $A$ and 
  any pair of 
  projective $A$-maximal fans $\Sigma$ and $\Sigma'$,
  the natural birational map 
  $\varphi: X_\Sigma \dashrightarrow X_{\Sigma'}$ (induced by the identity map of $N$) is decomposed into 
  a sequence of flips, flops and inverse flips.
  For any reduction $A\supsetdot A'$ and 
  any projective $A'$-maximal fan $\Sigma'$,
  there exists an extremal divisorial contraction 
  $X_{\Sigma} \rightarrow X_{\Sigma'}$ 
  such that $\Sigma$ is a projective $A$-maximal fan.
  For any Mori fiber structure $\Af \subsetm A$ with the base $\Ab$
  and any projective $\Ab$-maximal fan $\Sigmab$,
  there exists a toric Mori fiber space $X_\Sigma \rightarrow X_{\Sigmab}$ with fibers 
  isomorphic to $X_{\Sigmaf}$
  such that $\Sigma$ is a projective $A$-maximal fan and 
  $\Sigmaf\subset \Sigma$ is a projective $\Af$-maximal fan.
\end{lemma}
\begin{proof}
  All the statements in \pref{lm:lift} follow from the description of the movable cone 
  $\overline{\Mov}(X_\Sigma)$ of a $\bQ$-factorial projective toric variety $X_\Sigma$
  as the corresponding cone
  $\Mov_{\mathrm{GKZ}}$ in the secondary fan $\Sigma_{\mathrm{GKZ}}$
  defined by $A=G(\Sigma)$
  (see (15.1.5), Proposition 15.1.4 and Theorem 15.1.10(c) in \cite{MR2810322}).
  The first statement follows because
  $\Mov_{\mathrm{GKZ}}$ is a full-dimensional convex rational polyhedral cone
  decomposed into finite chambers corresponding to projective $A$-maximal fans.
  In fact, along a general path between the two chambers corresponding to $\Sigma$ and $\Sigma'$, 
  one obtains a sequence of flips, flops and inverse flips, which decomposes
  $\varphi: X_\Sigma \dashrightarrow X_{\Sigma'}$ (see \cite[Theorem 15.3.13]{MR2810322}).
  The second statement is a simple
  extension of \cite[Lemma 6.1]{fredrickson2015generalized} and the same proof works.
  That is, since $\Mov_{\mathrm{GKZ}}$ is full-dimensional and convex, there exists a chamber
  containing the cone corresponding to $\Sigma'$ as a face, which gives 
  an extremal divisorial contraction
  (refer to \cite[Lemma 14.4.6]{MR2810322} for the face structure of the secondary fan).
  Exactly the same proof works for the third statement 
  if one starts from a generalized fan 
  $\pi^{-1} \Sigmab \coloneqq \lc \pi_\bR^{-1}(\sigma) \relmid \sigma \in \Sigmab \rc$ instead of $\Sigma'$,
  where $\pi_\bR:N_\bR \rightarrow (\Nb)_\bR$ is the natural projection defined by \pref{eq:exact}.
\end{proof}

\begin{remark}
   \label{rm:shed}
  The \emph{shed} of a fan $\Sigma$
  is defined as 
  $\shed(\Sigma) = \bigcup_{\sigma\in \Sigma} 
  \Conv \lb \lc 0\rc \cup  G(\sigma) \rb \subset \Conv G(\Sigma)$.
  As observed originally by \cite[Proposition 4.3]{MR717617},
  for a birational contraction $\varphi_R: X_{\Sigma} \rightarrow X_{\Sigma'}$ of an extremal ray 
  $R\subset \overline{\NE}(X_\Sigma)$,
  $\shed(\Sigma') \subsetneq \shed(\Sigma)$ 
  (resp.\ $\shed(\Sigma') = \shed(\Sigma)$, $\shed(\Sigma') \supsetneq \shed(\Sigma)$)
  if and only if
  $K_{X_\Sigma}.R <0$ (resp.\ $K_{X_\Sigma}.R =0$, $K_{X_\Sigma}.R >0$).
  Be aware that the sign of $K_\Sigma.R$ 
  for a divisorial contraction $\varphi_R: X_{\Sigma} \rightarrow X_{\Sigma'}$
  actually depends not only on the primitive generating sets 
  $G(\Sigma)$ and $G(\Sigma')$ but
  also on $\shed(\Sigma')$ in general.
  However, for a reduction $\nabla \supsetdot \nabla'$ of Fano polytopes,
  any divisorial contraction $\varphi_R: X_{\Sigma} \rightarrow X_{\Sigma'}$ 
  with $\Conv G(\Sigma) = \nabla$ and $\Conv G(\Sigma') = \nabla'$
  satisfies $K_\Sigma.R < 0$ since $\shed(\Sigma') \subset \nabla'$ does not contain
  $v\in \nabla\setminus \nabla'$.
\end{remark}


\section{Description of toric Sarkisov links}
\label{sc:sarkisov}
A \emph{Sarkisov link} is one of the following four types of diagrams of normal projective varieties
that connects two Mori fiber spaces $p: X \rightarrow S$ and $p': X' \rightarrow S'$:
\begin{equation}
  \label{eq:links}
   \arraycolsep=8pt
  \begin{array}{cccc}
   \text{type I} & \text{type II} & \text{type III} & \text{type IV} \\
  \begin{tikzcd}[column sep=tiny]
    X \arrow[rr, dashrightarrow] \arrow[d, "p"'] && X'' \arrow[d]\\
    S \arrow[dr] && X' \arrow[ld, "p'"]\\
    & S'&
  \end{tikzcd} &
  \begin{tikzcd}[column sep=tiny]
    X'' \arrow[rr, dashrightarrow] \arrow[d] && X''' \arrow[d]\\
    X \arrow[d, "p"'] && X' \arrow[d, "p'"]\\
    S \arrow[rr, equal] & \phantom{.} &S'
  \end{tikzcd} &
  \begin{tikzcd}[column sep=tiny]
    X'' \arrow[rr, dashrightarrow] \arrow[d] && X' \arrow[d, "p'"]\\
    X \arrow[dr, "p"'] && S' \arrow[ld]\\
    & S&
  \end{tikzcd} &
  \begin{tikzcd}[column sep=tiny]
    X \arrow[rr, dashrightarrow] \arrow[d, "p"'] && X' \arrow[d, "p'"]\\
    S \arrow[dr] && S' \arrow[ld]\\
    & R&
  \end{tikzcd}
  \end{array}
\end{equation}
Here
all the varieties except $R$ are $\bQ$-factorial.
Every dashed arrow denotes a sequence of flips, flops and inverse flips,
and 
every solid arrow denotes an extremal contraction.
In the latter case, each solid arrow between $X$'s is a divisorial contraction of a $K$-negative extremal ray,  
each arrow between $S$'s is either divisorial or fibering contraction, and 
each arrow toward $R$ is either small or fibering contraction.

\begin{theorem}
Let $\Af \subsetm A$ and $\Af' \subsetm A'$ be Mori fiber primitive generating sets 
associated with toric Mori fiber spaces $p: X\rightarrow S$ and $p': X'\rightarrow S'$, respectively.
Suppose that 
these are connected
by a single Sarkisov link 
in \pref{eq:links}.
Then $\Af \subsetm A$ and $\Af' \subsetm A'$ 
coincide or
are related by one of the following diagrams and their inverses
(i.e., the diagrams switching $\Af\subsetm A$ and $\Af'\subsetm A'$):
\label{th:sarkisovA}
\begin{equation}
  \label{eq:sarkisovA}
   \arraycolsep=8pt\def\arraystretch{1.3}
  \begin{array}{ccccc}
   \text{type I}{}_\mathrm{d} & \text{type I}{}_\mathrm{m} & \text{type II}{}_{\mathrm{irr}} 
   & \text{type II}{}_{\mathrm{ni}}  &  \text{type IV}{}_ {\mathrm{m}}\\
  \begin{tikzcd}[row sep = small, column sep=-7pt]
    A & \supsetdot &  A'\\
    \Af \arrow[u, phantom, sloped, "\subsetm"]  & = & \Af'\arrow[u, phantom, sloped, "\subsetm"]
  \end{tikzcd} &
  \begin{tikzcd}[row sep = small, column sep=-7pt]
    A & = &  A'' & \supsetdot & A'\\
    \Af \arrow[u, phantom, sloped, "\subsetm"]
    &\subsetm& \Af''\arrow[u, phantom, sloped, "\subset"] & \supsetdot & \Af'\arrow[u, phantom, sloped, "\subsetm"]
  \end{tikzcd} & 
  \begin{tikzcd}[row sep = small, column sep=-7pt]
    A & \subsetdot &  A'' & \supsetdot & A'\\
    \Af \arrow[u, phantom, sloped, "\subsetm"]
    &\subsetdot & \Af''\arrow[u, phantom, sloped, "\subset"] & \supsetdot & \Af'\arrow[u, phantom, sloped, "\subsetm"]
  \end{tikzcd} & 
  \begin{tikzcd}[row sep = small, column sep=-7pt]
    A & \subsetdot &  A'' & \supsetdot & A'\\
    \Af \arrow[u, phantom, sloped, "\subsetm"]
    &=& \Af''\arrow[u, phantom, sloped, "\subset"] & = & \Af'\arrow[u, phantom, sloped, "\subsetm"]
  \end{tikzcd} & 
  \begin{tikzcd}[row sep = small, column sep=-7pt]
    A & = &  A'' & = & A'\\
    \Af \arrow[u, phantom, sloped, "\subsetm"]
    &\subsetm& \Af''\arrow[u, phantom, sloped, "\subset"] &\supsetm& 
    \Af'\arrow[u, phantom, sloped, "\subsetm"]
  \end{tikzcd}
  \end{array}
\end{equation}
\end{theorem}
We call one of the diagrams in \pref{eq:sarkisovA} and their inverses an \emph{elementary link} for
Mori fiber primitive generating sets.
The inverse of an elementary link of type I${}_d$ (resp.\ type I${}_m$) is referred to as 
that of type III${}_d$ (resp.\ type III${}_m$).
It is natural to add the trivial link, $A=A$ with $\Af=\Af'$, as an elementary link of type IV${}_\mathrm{s}$.
The scripts ``$\mathrm{s}$'', ``$\mathrm{d}$'', ``$\mathrm{m}$''
indicate that the contraction of $S$'s in the corresponding Sarkisov link is
a small contraction, a divisorial contraction, a Mori fiber space, respectively.
Also, ``$\mathrm{irr}$'' and ``$\mathrm{ni}$'' for type II links indicate 
that all fibers of the composite contractions $X''\rightarrow S$ and $X'''\rightarrow S'$ are irreducible, and
that the fibers are not irreducible over a common invariant prime divisor of the base $S=S'$, respectively.
\begin{proof}
We begin with describing the composition of two extremal contractions $q$ and $q'$,
\begin{equation}
 \begin{tikzcd}
   Y \arrow[r, "q"] & Y' \arrow[r, "q'"] & Y''
 \end{tikzcd}
\end{equation}
appeared in one of the Sarkisov links in \pref{eq:links} for 
$\bQ$-factorial projective toric varieties. 
There are four types of such compositions:
\begin{enumerate}
   \item \label{it:1} $q$ is a fibering contraction and $q'$ is a small contraction,
   \item \label{it:2} $q$ is a fibering contraction and $q'$ is a divisorial contraction,
   \item \label{it:3} $q$ is a divisorial contraction and $q'$ is a fibering contraction,
   \item \label{it:4} both $q$ and $q'$ are fibering contractions.
\end{enumerate}
By \pref{lm:contractions}, 
the contractions $q$ and $q'$ are described by primitive generating sets as 
\begin{enumerate}
   \item[\pref{it:1}] $\Bf\subsetm B$ and $\Bb=\Bb'$,
   \item[\pref{it:2}] $\Bf\subsetm B$ and $\Bb\supsetdot \Bb'$,
   \item[\pref{it:3}] $B\supsetdot B'$ and $\Bf' \subsetm B'$,
   \item[\pref{it:4}] $\Bf \subsetm B$ and $B_{\mathrm{b, f}}\subsetm \Bb$,
\end{enumerate}
for each type of composition, respectively. 
We shall combine two of such descriptions into
an elementary link for Mori fiber primitive generating sets.
Hereafter, we denote by
$\Af \subsetm A$ and $\Af' \subsetm A'$  
the two Mori fiber structures
corresponding to $p$ and $p'$ 
in each Sarkisov link,
respectively.

A composition of type \pref{it:1} only appears in the Sarkisov link of type IV${}_\mathrm{s}$.
In this Sarkisov link, 
both compositions 
are of type \pref{it:1}
so that $A = A'=B$ and $\Ab = \Ab'=\Bb$ by \pref{lm:contractions}, and hence, $\Af = \Af'$ as well.
This yields the trivial elementary link of type IV${}_\mathrm{s}$.

A composition of type \pref{it:2} only appears in the Sarkisov link of type I${}_\mathrm{d}$ 
(or type III${}_\mathrm{d}$ as its inverse).
In this case, we have a reduction $B\supsetdot B'\coloneqq \pi^{-1}(\Bb'\cup \lc0 \rc)\cap B$
since 
the fiber structure $\Bf\subsetm B$ is irreducible so that
$B\setminus B' = \pi^{-1}(\Bb\setminus \Bb')\cap B$ consists of exactly one point.
Thus we obtain the elementary link of type I${}_\mathrm{d}$,
\begin{equation}
   \label{eq:diagram1}
   \begin{tikzcd}[row sep = small, column sep=-5pt]
      B & \supsetdot & B'\\
      \Bf \arrow[u, phantom, sloped, "\subsetm"]  & = & \Bf'\arrow[u, phantom, sloped, "\subsetm"]
   \end{tikzcd},
\end{equation}
with $\Bb \supsetdot \Bb'$ and 
$\Bf'\coloneqq B'\cap L$, where $L$ is the linear span of $\Bf$.
Clearly, the fiber structure $\Bf' \subset B'$ is irreducible and 
its base coincides with $\Bb'$ given in advance.
We may put $A=B$ and $A' = B'$ (resp.\ $A=B'$ and $A' =B$)
for the Sarkisov link of type I${}_\mathrm{d}$ (resp.\ type III${}_\mathrm{d}$).
In either case, the opponent composition is of type \pref{it:3} below.

For a type \pref{it:3} composition, we have a sequence 
$B \supsetdot B' \supsetm \Bf'$.  Let $v$ be the unique element in 
$B\setminus B'$ and $L'$ be the linear span of $\Bf'$.
Then, 
depending on whether $v \not \in L'$ or $v\in L'$, 
we can make the following diagrams, respectively:
\begin{equation}
   \label{eq:diagram2}
   \begin{tikzcd}[row sep = small, column sep=-5pt]
      B & \supsetdot & B'\\
      \Bf \arrow[u, phantom, sloped, "\subset"]  & = & \Bf'\arrow[u, phantom, sloped, "\subsetm"]
   \end{tikzcd}, \quad \text{or} \quad
   \begin{tikzcd}[row sep = small, column sep=-5pt]
      B & \supsetdot & B'\\
      \Bf \arrow[u, phantom, sloped, "\subset"]  & \supsetdot & \Bf'\arrow[u, phantom, sloped, "\subsetm"]
   \end{tikzcd}.
\end{equation}
In fact, 
in the case of 
$v \not\in L'$,
it is clear that $\Bf\coloneqq \Bf' \subset B$ defines a fiber structure. 
Furthermore, if the fiber structure is irreducible, 
the left diagram in \pref{eq:diagram2}
coincides with the elementary link of type I${}_\mathrm{d}$ \pref{eq:diagram1},
which is already discussed.
If the fiber structure is not irreducible, 
we call the diagram a \emph{non-irreducible half-link},
which should be coupled with another half-link. 
In the case of $v\in L'$, 
we have a reduction
$\Bf\coloneqq B\cap L' \supsetdot \Bf'$ as in
the right diagram in \pref{eq:diagram2}, and
$\Bf \subset B$ becomes an irreducible fiber structure.
We call the diagram an \emph{irreducible half-link I}.

Finally, in the case of a type \pref{it:4}, 
we denote by $\Bb'$ the base of the fiber structure $B_\mathrm{b, f}\subsetm \Bb$.
Accordingly, we set a diagram of lattices and fit all relevant sets of lattice points together:
\begin{equation}
   \begin{tikzcd}[row sep = small]
      N \arrow[r, twoheadrightarrow, "\pi"] & \Nb \arrow[r, twoheadrightarrow, "\pi_\mathrm{b}"] & \Nb' \\
      \Nf' \arrow[u, phantom, sloped, "\subset"]  \arrow[r, twoheadrightarrow] &  N_{\mathrm{b,f}} 
      \arrow[u, phantom, sloped, "\subset"]  & \\
      \Nf\arrow[u, phantom, sloped, "\subset"] & &
   \end{tikzcd}
   , \qquad
   \begin{tikzcd}[row sep = small]
      B \arrow[r, dashrightarrow, "\pibar"]& \Bb \arrow[r, dashrightarrow, "\overline{\pi}_\mathrm{b}"] & \Bb' \\
      \Bf' \arrow[r, dashrightarrow] \arrow[u, phantom, sloped, "\subset"]  &  B_{\mathrm{b,f}} 
      \arrow[u, phantom, sloped, "\subsetm"] & \\
      \Bf\arrow[u, phantom, sloped, "\subset"] &&
   \end{tikzcd},
\end{equation}
where $\Nf' = \ker(\pi_\mathrm{b}\circ \pi)$ and $\Bf' = B \cap \Nf'$.
Since 
$B_{\mathrm{b,f}}\subset N_{\mathrm{b,f}}$ is a primitive generating set,
so is $\Bf'\subset \Nf'$, and hence,
$\Bf'\subset B$ and $\Bf \subset \Bf'$ are fiber structures.
By definition, Mori fiber structures $\Bf \subsetm B$ and $B_{\mathrm{b,f}}\subsetm \Bb$
are irreducible, and hence, the equalities
$\abs{B}=\abs{\Bf}+\abs{\Bb} = \abs{\Bf} + \abs{B_{\mathrm{b,f}}} + \abs{\Bb'}$ hold.
Then the inequalities
$\abs{B} \ge \abs{\Bf'} + \abs{\Bb'}$ and $\abs{\Bf'} \ge \abs{\Bf} + \abs{B_{\mathrm{b,f}}}$
for fiber structures 
$\Bf'\subset B$ and $\Bf \subset \Bf'$
must be equalities, that is,
both $\Bf'\subset B$ and $\Bf \subset \Bf'$ are also irreducible. 
In particular, $\Bf \subset \Bf'$ is a Mori fiber structure.
Thus we obtain another diagram:
\begin{equation}
   \label{eq:diagram3}  
   \begin{tikzcd}[row sep = small, column sep=-5pt]
      B & = & B'\\
      \Bf \arrow[u, phantom, sloped, "\subsetm"]  & \subsetm & \Bf'\arrow[u, phantom, sloped, "\subset"]
   \end{tikzcd}.
\end{equation}
We call the diagram an \emph{irreducible half-link II}.

Now we combine two of half-links to compose remaining elementary links
between $\Af \subsetm A$ and $\Af'\subsetm A'$, which is very easy.
If the fiber structure in the middle is not irreducible, 
there is only one way to combine them, 
resulting the elementary link of type II${}_\mathrm{ni}$ in \pref{eq:sarkisovA}.
If the fiber structure in the middle is irreducible,
one has the remaining three ways,
I-I, I-II, and II-II, 
except switching how to set $\Af \subsetm A$ and $\Af'\subsetm A'$,
which results the remaining elementary links 
of type II${}_\mathrm{irr}$, type I${}_\mathrm{m}$, and type IV${}_\mathrm{m}$
in \pref{eq:sarkisovA}, respectively.
This completes the proof.
\end{proof}

\begin{corollary}
   \label{cr:seqlinks}
   For any pair of Mori fiber primitive generating sets 
   $\Af \subsetm A$ and $\Af' \subsetm A'$ in $N_\bR$,
   there exists a sequence of elementary links connecting them.
\end{corollary}
\begin{proof}
   By \pref{lm:lift}, there exists a pair of toric Mori fiber spaces 
   $p: X\rightarrow S$ and $p': X' \rightarrow S'$ which represent 
   $\Af \subsetm A$ and $\Af' \subsetm A'$, respectively.
   The $p$ and $p'$ are connected by a sequence of toric Sarkisov links by the Sarkisov program for toric varieties 
   \cite[Section 14.5]{MR1875410}, \cite{MR3019454}.
   By coarse-graining the resulting sequence of Sarkisov links, one obtains a sequence of elementary links connecting 
   $\Af \subsetm A$ and $\Af' \subsetm A'$ by \pref{th:sarkisovA}.
\end{proof}

We define a \emph{log Sarkisov link} as one of the four types of diagrams in \pref{eq:links}
but each divisorial contraction of $X$'s is not imposed to be $K$-negative.
\begin{corollary}
   \label{cr:sarkisovlift}
   For any sequence of elementary links connecting two Mori fiber primitive generating sets
   $\Af\subsetm A$ and $\Af' \subsetm A'$,
   there exists a sequence of log Sarkisov links connecting two toric Mori fiber spaces 
   $p: X\rightarrow S$ and 
   $p': X'\rightarrow S'$
   which represent
   $\Af\subsetm A$ and $\Af' \subsetm A'$, respectively.
\end{corollary}
\begin{proof}
Assume $\Af =\Af'$ and $A=A'$ first. 
There is a natural birational map
$\varphi_\mathrm{b}: S\dashrightarrow S'$ induced by the identity map of $\Nb$.
By the first statement of \pref{lm:lift},
$\varphi_\mathrm{b}$ is decomposed into flips, flops, and inverse flips.
Since each intermediate birational model $S''$ also represents the base $\Ab$ of $\Af\subsetm A$,
there exists a toric Mori fiber space $p'': X''\rightarrow S''$ which represent 
$\Af\subsetm A$ by the third statement of \pref{lm:lift}.
Hence, by replacing $p'$ with $p''$, 
one may assume $\varphi_\mathrm{b}$ is either a single flip, flop, or inverse flip.
In either case, it fits a log Sarkisov link of type IV${}_\mathrm{s}$ because
the natural map $\varphi: X\dashrightarrow X'$ is decomposed into flips, flops, and inverse flips
and $\varphi_\mathrm{b}\circ p = p' \circ \varphi$ holds.

Next, assume that
$\Af\subsetm A$ and $\Af' \subsetm A'$ are related by a single elementary link.
As in the first case, 
one can make a log Sarkisov link from the bottom of its diagram 
for each elementary link by applying \pref{lm:lift} repeatedly. For this purpose, 
it is helpful to rewrite the elementary links in \pref{eq:sarkisovA} into 
alternative descriptions by using the bases of fiber structures:
\begin{equation}
  \label{eq:sarkisovA2}
   \arraycolsep=9pt\def\arraystretch{1.3}
  \begin{array}{cccc}
   \text{type I}{}_\mathrm{d} & \text{type I}{}_\mathrm{m} & \text{type II}{}_{\mathrm{irr}} 
   \text{/II}{}_{\mathrm{ni}}  &  \text{type IV}{}_ {\mathrm{m}}\\
  \begin{tikzcd}[row sep = 15pt, column sep=-3pt]
    A \arrow[d,dashrightarrow] & \supsetdot &  A'\arrow[d,dashrightarrow] \\
    \Ab  & \supsetdot & \Ab'
  \end{tikzcd} &
  \begin{tikzcd}[row sep = 15pt, column sep=-7pt]
    A \arrow[d,dashrightarrow]& = &  A'' \arrow[d,dashrightarrow]& \supsetdot & A'\arrow[d,dashrightarrow]\\
    \Ab
    &\dashrightarrow & \Ab'' & \begin{minipage}{20pt}\begin{center}=\end{center}\end{minipage} & \Ab'
  \end{tikzcd} & 
  \begin{tikzcd}[row sep = 15pt, column sep=-3pt]
    A \arrow[d,dashrightarrow]& \subsetdot &  A'' \arrow[d,dashrightarrow]& \supsetdot & A'\arrow[d,dashrightarrow]\\
    \Ab
    &=& \Ab''
    & = & \Ab'
  \end{tikzcd} & 
  \begin{tikzcd}[row sep = 15pt, column sep=-7pt]
    A \arrow[d,dashrightarrow]& = &  A'' \arrow[d,dashrightarrow]& = & A'\arrow[d,dashrightarrow]\\
    \Ab &\dashrightarrow & \Ab''
    &\dashleftarrow & \Ab'
  \end{tikzcd}
  \end{array}
\end{equation}

Finally, in general case, the log Sarkisov links obtained from each elementary link
are glued together by the sequence of log Sarkisov links of type IV${}_\mathrm{s}$ in the first case, 
so that one obtains a sequence of log Sarkisov links connecting $p: X \rightarrow S$ and $p':X'\rightarrow S'$.
\end{proof}

\begin{corollary}
   \label{cr:polytopelinks}
   For any sequence of elementary links connecting two Mori fiber polytopes
   $\nablaf\subsetm \nabla$ and $\nablaf' \subsetm \nabla'$
   consisting of only Fano polytopes, 
   there exists a sequence of Sarkisov links connecting two toric Mori fiber spaces 
   $p: X\rightarrow S$ and 
   $p': X'\rightarrow S'$
   which represent
   $\nablaf\subsetm \nabla$ and $\nablaf' \subsetm \nabla'$, respectively.
\end{corollary}
\begin{proof}
This is an immediate consequence of \pref{cr:sarkisovlift}
and the last sentence of \pref{rm:shed}.
\end{proof}

In order to establish the whole picture 
only in terms of Fano polytopes,
the following two problems could be keys, 
the latter of which is regarded as the Fano polytope version of \pref{cr:seqlinks}.  
We say that a Fano polytope $\nabla$ is \emph{minimal} if there is no reduction $\nabla \supsetdot \nabla'$.
\begin{problem}
   \label{pb:polytopeMMP}
   For any minimal Fano polytope $\nabla$,  is there a Mori fiber structure $\nablaf\subsetm \nabla$?
\end{problem}
\begin{problem}
   \label{pb:mfpconnectedness}
   For any pair of Mori fiber polytopes
   $\nablaf \subsetm \nabla$ and $\nablaf' \subsetm \nabla'$ in $N_\bR$,
   is there a sequence of elementary links consisting only of Fano polytopes 
   which connects $\nablaf\subsetm \nabla$ and $\nablaf' \subsetm \nabla'$?
\end{problem}

\begin{remark}
\pref{pb:polytopeMMP} is required to guarantee that the MMP 
works only inside the category of Fano polytopes.
Namely, it assures that
one ends up with a Mori fiber polytope
after finite times of reductions of a given Fano polytope.
Note that, by triangulating a minimal Fano polytope $\nabla$ using only its vertices,
one can always find a fiber structure $\nablaf \subset \nabla$ with 
a minimal Fano simplex $\nablaf$ (see \cite[Proposition 3.2]{MR2760660}, and also \cite[Lemma 1]{MR1901220} 
for the structure of minimal primitive generating sets).
Hence,  one part of the problem is whether 
a minimal Fano simplex $\nablaf$ allows a fiber structure 
$\nablaf'\subset \nablaf$ such that $\nablaf'$ is a terminal simplex.
\end{remark}

The following example shows the subtleties of \pref{pb:mfpconnectedness} in three or higher dimensions.
The obstacle stems from the fact that, for a relevant Fano polytope $\nabla$, 
a projective simplicial terminal fan $\Sigma$ satisfying $\Conv G(\Sigma)=\nabla$
needs not to be a projective $\nabla$-maximal fan.

\begin{example}
\label{eg:obstacle}
Set
$v_1 = (1,0,0)$,
$v_2 = (0,1,0)$,
$v_3 = (-1,-1,0)$,
$v_4 = (-1,0,0)$,
$v_5 = (0,0,1)$, 
$v_6 = (-1,0,-1)$, and
$v_7= (-2,0,-1)$
in $N\simeq \bZ^3$.
As an abbreviation, we write 
Fano polytopes
$\nabla_{i_1\cdots i_r} = \Conv(v_{i_1}, \dots, v_{i_r})$
and 
primitive generating sets
$A_{i_1\cdots i_r} = \lc v_{i_1}, \dots, v_{i_r} \rc$
with concatenated subscripts for any indices $i_1, \dots, i_r$.
In the following,
the latter symbol is used if and only 
if $A_{i_1\cdots i_r} \ne \nabla_{i_1\cdots i_r}\cap \Nprim$.
Then we have a sequence of elementary links
connecting a pair of Mori fiber polytopes $\nabla_{14}\subsetm \nabla_{123457}$
and $\nabla_{123}\subsetm \nabla_{12356}$,
\begin{equation}
   \label{eq:fano}
   \begin{tikzcd}[row sep = small, column sep=-5pt]
    \nabla_{123457} & = & \nabla_{123457} & \supsetdot & A_{12357} & \subsetdot
    &A_{123567} & \supsetdot & \nabla_{12356}\\
    \nabla_{14} \arrow[u, phantom, sloped, "\subsetm"]  
    & \subsetm & \nabla_{1234} \arrow[u, phantom, sloped, "\subset"]  
    & \supsetdot & \nabla_{123} \arrow[u, phantom, sloped, "\subsetm"]  
    & = & \nabla_{123} \arrow[u, phantom, sloped, "\subset"]
    & = & \nabla_{123} \arrow[u, phantom, sloped, "\subsetm"]
  \end{tikzcd}.
\end{equation}
The sequence  \pref{eq:fano} corresponds to that of Sarkisov links between toric Mori fiber spaces, 
which are all projective bundles (cf.\ \cite[Proposition 7.3.3]{MR2810322}),
\begin{equation}
   \begin{tikzcd}
   \bP_{\bP^1\times \bP^1}\lb \cO \oplus \cO(-1,-2) \rb \arrow[r,dashrightarrow,"\text{I}_{\mathrm{d}}"] &
   \bP_{\bP^1}\lb \cO \oplus \cO \oplus  \cO(-2) \rb \arrow[r,dashrightarrow,"\text{II}_{\mathrm{ni}}"] &
   \bP_{\bP^1}\lb \cO \oplus \cO \oplus  \cO(-1) \rb,
   \end{tikzcd}
\end{equation}
where we put the types of the Sarkisov links on dashed arrows.
Note that the coarse-graining of the sequence \pref{eq:fano} into Fano polytopes does not attain
a sequence of elementary links.
On the other hand, one can also take another sequence of elementary links 
connecting the same pair of Mori fiber polytopes
which consists only of Fano polytopes:
\begin{equation}
   \label{eq:fano2}
   \begin{tikzcd}[row sep = small, column sep=-5pt]
    \nabla_{123457} & \subsetdot & \nabla_{1234567} & \supsetdot & \nabla_{123456} & =
    &\nabla_{123456} & \supsetdot & \nabla_{12356}\\
    \nabla_{14} \arrow[u, phantom, sloped, "\subsetm"]  
    & = & \nabla_{14} \arrow[u, phantom, sloped, "\subset"]  
    & = & \nabla_{14} \arrow[u, phantom, sloped, "\subsetm"]  
    & \subsetm & \nabla_{1234} \arrow[u, phantom, sloped, "\subset"]
    & \supsetdot & \nabla_{123} \arrow[u, phantom, sloped, "\subsetm"]
  \end{tikzcd}.
\end{equation}
The sequence \pref{eq:fano2} corresponds to the following sequence of Sarkisov links:
\begin{equation}
   \begin{tikzcd}
   \bP_{\bP^1\times \bP^1}\lb \cO \oplus \cO(-1,-2) \rb \arrow[r,dashrightarrow,"\text{II}_{\mathrm{ni}}"] &
   \bP_{\bP^1\times \bP^1}\lb \cO \oplus  \cO(-1, -1) \rb \arrow[r,dashrightarrow,"\text{I}_{\mathrm{d}}"] &
   \bP_{\bP^1}\lb \cO \oplus \cO \oplus  \cO(-1) \rb.
   \end{tikzcd}
\end{equation}
\end{example}

Now, let us see that the situation becomes better
in two dimensions
once we use the classification of 
the Sarkisov links for smooth rational Mori fiber surfaces:
\begin{equation}
  \label{eq:2dimlinks}
   \arraycolsep=8pt \def\arraystretch{1.2}
  \begin{array}{cccc}
   \text{type I}_\mathrm{m} & \text{type II}_\mathrm{ni} & \text{type III}_\mathrm{m} & 
   \text{type IV}_\mathrm{m} \vspace{5pt}\\
  \begin{tikzcd}[column sep= 20pt, row sep=20pt]
    \bF_1 \arrow[d] \arrow[dr]\\
    \bP^1 \arrow[dr] & \bP^2 \arrow[d]\\
    & pt
  \end{tikzcd} &
  \begin{tikzcd}[column sep=10pt, row sep=20pt]
    & S_x \arrow[ld, "\mathrm{bl}_x"'] \arrow[dr]&\\
    \bF_m \arrow[dr] && \bF_{m\pm 1} \arrow[dl]\\
    & \bP_1&
  \end{tikzcd} &
  \begin{tikzcd}[column sep=20pt, row sep=20pt]
    & \bF_1 \arrow[d] \arrow[dl]\\
    \bP^2 \arrow[d] & \bP^1 \arrow[dl]\\
    pt&
  \end{tikzcd} &
  \begin{tikzcd}[column sep=0pt, row sep=20pt]
    &\bP^1 \times \bP^1 \arrow[dl, "p_1"'] \arrow[dr, "p_2"]&\\
    \bP^1 \arrow[dr] && \bP^1 \arrow[ld]\\
    & pt&
  \end{tikzcd}
  \end{array}
\end{equation}
Here $\bF_m = \bP_{\bP^1}\lb \cO \oplus \cO(-m) \rb$ is the Hirzebruch surface for $m\ge 0$, 
$p_i: \bF_0 = \bP^1\times \bP^1\rightarrow \bP^1$ is the $i$-th projection ($i=1,2$),
and $\mathrm{bl}_x: S_x\rightarrow \bF_m$ 
is a blow-up at a point $x\in \bF_m$. 
The Sarkisov link of type II${}_\mathrm{ni}$ in \pref{eq:2dimlinks} is classically known as an 
\emph{elementary transform}, that is, 
a blow-up at a point $x\in \bF_m$ followed by the contraction of the strict transform of the fiber 
passing through $x$
with respect to a ruling $\bF_m\rightarrow \bP^1$ \cite[Example 5.7.1]{MR0463157}. 
We denote it by
\begin{equation}
   \label{eq:elem}
   \el_x: \bF_m \dashrightarrow \bF_{m\pm 1}
\end{equation} 
in the following.
Note that, for $m\ge 1$, 
the elementary transform at $x\in \bF_m$ results $\bF_{m+1}$ or $\bF_{m-1}$
depending on whether $x$ lies on the section $D\subset \bF_m$ with negative self-intersection, $D^2=-m$.

\begin{proposition}
   \label{pr:twodim}
   \pref{pb:polytopeMMP} and 
   \pref{pb:mfpconnectedness} are affirmative in two dimensions.
\end{proposition}
\begin{proof}
It is easy to verify that
all smooth toric Mori fiber surfaces $\bP^2$ and $\bF_m$ ($m\ge 0$) are described by Mori fiber polygons,
and
all toric Sarkisov links in \pref{eq:2dimlinks}
are described by elementary links consisting only of Fano polygons (see \pref{sc:reflexive} for concrete forms).
\end{proof}

\section{Two Proofs of \pref{th:polygon}}
\label{sc:reflexive}
In this section, 
we give a constructive proof of \pref{th:polygon} first, and then a non-constructive proof of it.
The latter proof uses a stronger claim (\pref{lm:smallseq}), which also implies 
the connectedness of the web of elliptic elephants in smooth projective rational surfaces 
(\pref{cr:ellip}).

\begin{proof}[{Constructive proof of \pref{th:polygon}}]
   Recall that reflexive is equivalent to canonical for lattice polygons. 
   By repeating reductions for a canonical polygon, 
   one ends up with a canonical Mori fiber polygon by \pref{pr:twodim}.
   Then it suffices to show that any pair of canonical Mori fiber polygons 
   is connected via a sequence of elementary links whose intermediate polygons are all canonical.
   The same is true for terminal polygons.
   Up to unimodular equivalences, there are only 
   three terminal Mori fiber polygons $\nabla_{-\infty}$, $\nabla_0$ and $\nabla_1$ and 
   four canonical Mori fiber polygons $\nabla_{-\infty}$, $\nabla_0$, $\nabla_1$ and $\nabla_2$.
   Here, for any integer $m\ge 0$
   and a fixed lattice basis $\lc e_1, e_2 \rc$ of $N\simeq \bZ^2$, we set
   \begin{equation}
      \label{eq:polygons}
      \nabla_{-\infty} \coloneqq \Conv (e_1, e_2, -e_1-e_2) \quad 
      \text{and} \quad \nabla_m\coloneqq \Conv(e_1, e_2, -e_1, -m e_1-e_2)  
   \end{equation}
   corresponding to $\bP^2$ and $\bF_m$, respectively.
   There is an elementary link of type II${}_\mathrm{ni}$ connecting $\nabla_{m}$ and $\nabla_{m+1}$,
   \begin{equation}
      \label{eq:linkm}
      \begin{tikzcd}[column sep = -7pt, row sep = small]
         \nabla_{m} & \subsetdot \  & \Conv(\nabla_{m}\cup \nabla_{m+1}) & \  \supsetdot & \nabla_{m+1}\\
         \Conv(\pm e_1)  \arrow[u, phantom, sloped, "\subsetm"] & = & \Conv(\pm e_1)  \arrow[u, phantom, sloped, "\subset"] 
         & = &\Conv(\pm e_1) \arrow[u, phantom, sloped, "\subsetm"].
      \end{tikzcd}
   \end{equation}
   The elementary link \pref{eq:linkm}
   and its inverse are denoted by $\ell_m^{+}$ and $\ell_m^{-}$, respectively.
   It is easy to observe that all polygons in $\ell_0^{\pm}$ are terminal and 
   all polygons in $\ell_1^{\pm}$ are canonical.
   By \pref{lm:unimodular} below,
   we may assume each Mori fiber polygon is moved to standard form \pref{eq:polygons} in advance. 
   By using $\ell_m^\pm$ and another elementary link $\ell_{-\infty}^+$ of type III${}_\mathrm{m}$,
   \begin{equation}
      \label{eq:link0}
      \begin{tikzcd}[column sep = -7pt, row sep = small]
         \nabla_{-\infty} & \subsetdot \  & \nabla_{1} & \  = & \nabla_{1}\\
         \nabla_{-\infty}  \arrow[u, phantom, sloped, "\subsetm"] & \subsetdot & \nabla_{1}  \arrow[u, phantom, sloped, "\subset"] 
         & \supsetm &\Conv(\pm e_1) \arrow[u, phantom, sloped, "\subsetm"],
      \end{tikzcd}
   \end{equation}
   with its inverse $\ell_{-\infty}^-$,
   each pair of terminal Mori fiber polygons from $\lc \nabla_{-\infty}, \nabla_0, \nabla_1 \rc$ 
   (resp.\ canonical Mori fiber polygons from $\lc \nabla_{-\infty}, \nabla_0, \nabla_1, \nabla_2 \rc$)
   can be connected
   via a sequence of elementary links whose intermediate polygons are all terminal (resp.\ canonical).
   This completes the proof.
\end{proof}

\begin{lemma}
   \label{lm:unimodular}
   With the same notation as above,
   any unimodular transformation of $\nabla_{-\infty}$ (resp.\ $\nabla_m$ for $m\ge 0$) is decomposed into
   a sequence of elementary links unimodular equivalent to $\ell^\pm$, $\ell_{-\infty}^\pm$ and $\ell_0^\pm$ (resp.\ 
   $\ell^{\pm}$, $\ell_j^\pm$ with $0\le j\le \max\lc 1,m-1 \rc$), where $\ell^+$ is the following 
   elementary link of type IV${}_\mathrm{m}$,
   \begin{equation}
      \begin{tikzcd}[column sep = -7pt, row sep = small]
         \nabla_0 & = \  & \nabla_0 & \  = & \nabla_0\\
         \Conv(\pm e_1)  \arrow[u, phantom, sloped, "\subsetm"] & \subsetm & \nabla_0  
         \arrow[u, phantom, sloped, "\subset"] 
         & \supsetm &\Conv(\pm e_2) \arrow[u, phantom, sloped, "\subsetm"],
      \end{tikzcd}
   \end{equation}
   and $\ell^-$ is its inverse.
   In particular, any terminal (resp.\ canonical) Mori fiber polygons is connected to 
   that in standard form \pref{eq:polygons}
   through a sequence of elementary links whose intermediate polygons are all terminal (resp.\ canonical).
\end{lemma}
\begin{proof}
It is enough to show the claim only for 
the pairs of polygons, either 
$\nabla_{-\infty}$ or $\nabla_m$ in standard form \pref{eq:polygons}
and one of those
moved by the action of a generator of $GL(N) \simeq GL(2,\bZ)$, say,
\begin{equation}
   S = \begin{pmatrix}
      0 & -1 \\
      1 & 0
   \end{pmatrix}, \quad
   T = \begin{pmatrix}
      1 & 1 \\
      0 & 1
   \end{pmatrix} \quad \text{or} \quad
   U = \begin{pmatrix}
      -1 & 0 \\
      0 & 1
   \end{pmatrix}.
\end{equation}
Indeed, any other pair of transformed polygons 
is related by a finite composition of some conjugates of the above moves,
and, for such a conjugate move, it can be decomposed into a sequence of 
the corresponding conjugates of the elementary links.
Thus the claim can be verified by direct computation.
For example, the pair $\nabla_{-\infty}$ and $S\nabla_{-\infty}$ is connected by a sequence,
\begin{equation}
   \label{eq:cremona}
    \begin{tikzpicture}[scale=0.5, baseline={(0,0)}]
    \fill[gray] (1,0) -- (0,1) -- (-1,-1) --  cycle;
    \draw[black] (1,0) -- (0,1) -- (-1,-1) --  cycle;
    \fill (0,0) circle (2.0pt); 
    \fill (1,0) circle (2.0pt); 
    \fill (0,1) circle (2.0pt); 
    \fill (-1,-1) circle (2.0pt); 
    \node at (0,-1.5) {$\scriptscriptstyle{\mathrm{Mfp}}$};
    \end{tikzpicture} 
 \,\subsetdot \,
    \begin{tikzpicture}[scale=0.5, baseline={(0,0)}]
    \fill[gray] (1,0) -- (0,1) -- (-1,0) -- (-1,-1) --  cycle;
    \draw[black] (1,0) -- (0,1) -- (-1, 0) --(-1,-1) --  cycle;
    \fill (0,0) circle (2.0pt); 
    \fill (1,0) circle (2.0pt); 
    \fill (0,1) circle (2.0pt); 
    \fill (-1,0) circle (2.0pt); 
    \fill (-1,-1) circle (2.0pt); 
    \end{tikzpicture} 
 \,= \,
    \begin{tikzpicture}[scale=0.5, baseline={(0,0)}]
    \draw[black] (1,0) -- (0,1) -- (-1, 0) --(-1,-1) --  cycle;
    \draw[thick, gray] (0,0) -- (1,0);
    \draw[thick, gray] (0,0) -- (-1,0);
    \fill (0,0) circle (2.0pt); 
    \fill (1,0) circle (2.0pt); 
    \fill (0,1) circle (2.0pt); 
    \fill (-1,0) circle (2.0pt); 
    \fill (-1,-1) circle (2.0pt); 
    \node at (0,-1.5) {$\scriptscriptstyle{\mathrm{Mfp}}$};
    \end{tikzpicture} 
 \,\subsetdot \,
    \begin{tikzpicture}[scale=0.5, baseline={(0,0)}]
    \draw[black] (1,0) -- (0,1) -- (-1, 0) --(-1,-1) -- (0,-1) -- cycle;
    \draw[thick, gray] (0,0) -- (1,0);
    \draw[thick, gray] (0,0) -- (-1,0);
    \fill (0,0) circle (2.0pt); 
    \fill (1,0) circle (2.0pt); 
    \fill (0,1) circle (2.0pt); 
    \fill (0,-1) circle (2.0pt); 
    \fill (-1,0) circle (2.0pt); 
    \fill (-1,-1) circle (2.0pt); 
    \end{tikzpicture} 
 \,\supsetdot \,
    \begin{tikzpicture}[scale=0.5, baseline={(0,0)}]
    \draw[black] (1,0) -- (0,1) -- (-1, 0) -- (0,-1) -- cycle;
    \draw[thick, gray] (0,0) -- (1,0);
    \draw[thick, gray] (0,0) -- (-1,0);
    \fill (0,0) circle (2.0pt); 
    \fill (1,0) circle (2.0pt); 
    \fill (0,1) circle (2.0pt); 
    \fill (0,-1) circle (2.0pt); 
    \fill (-1,0) circle (2.0pt); 
    \node at (0,-1.5) {$\scriptscriptstyle{\mathrm{Mfp}}$};
    \end{tikzpicture} 
 \,\subsetdot \,
    \begin{tikzpicture}[scale=0.5, baseline={(0,0)}]
    \draw[black] (1,0) -- (0,1) -- (-1, 0) -- (0,-1) -- (1,-1) -- cycle;
    \draw[thick, gray] (0,0) -- (1,0);
    \draw[thick, gray] (0,0) -- (-1,0);
    \fill (0,0) circle (2.0pt); 
    \fill (1,0) circle (2.0pt); 
    \fill (0,1) circle (2.0pt); 
    \fill (0,-1) circle (2.0pt); 
    \fill (-1,0) circle (2.0pt); 
    \fill (1,-1) circle (2.0pt); 
    \end{tikzpicture} 
 \,\supsetdot \,
    \begin{tikzpicture}[scale=0.5, baseline={(0,0)}]
    \draw[black] (1,0) -- (0,1) -- (-1, 0) -- (1,-1) -- cycle;
    \draw[thick, gray] (0,0) -- (1,0);
    \draw[thick, gray] (0,0) -- (-1,0);
    \fill (0,0) circle (2.0pt); 
    \fill (1,0) circle (2.0pt); 
    \fill (0,1) circle (2.0pt); 
    \fill (-1,0) circle (2.0pt); 
    \fill (1,-1) circle (2.0pt); 
    \node at (0,-1.5) {$\scriptscriptstyle{\mathrm{Mfp}}$};
    \end{tikzpicture} 
 \,= \,
    \begin{tikzpicture}[scale=0.5, baseline={(0,0)}]
    \fill[gray] (1,0) -- (0,1) -- (-1,0) -- (1,-1) --  cycle;
    \draw[black] (1,0) -- (0,1) -- (-1, 0) -- (1,-1) -- cycle;
    \fill (0,0) circle (2.0pt); 
    \fill (1,0) circle (2.0pt); 
    \fill (0,1) circle (2.0pt); 
    \fill (-1,0) circle (2.0pt); 
    \fill (1,-1) circle (2.0pt); 
    \end{tikzpicture} 
 \,\supsetdot \,
    \begin{tikzpicture}[scale=0.5, baseline={(0,0)}]
    \fill[gray]  (0,1) -- (-1,0) -- (1,-1) --  cycle;
    \draw[black] (0,1) -- (-1, 0) -- (1,-1) -- cycle;
    \fill (0,0) circle (2.0pt); 
    \fill (0,1) circle (2.0pt); 
    \fill (-1,0) circle (2.0pt); 
    \fill (1,-1) circle (2.0pt); 
    \node at (0,-1.5) {$\scriptscriptstyle{\mathrm{Mfp}}$};
    \end{tikzpicture},
\end{equation}
where we encode the fiber structure $\nablaf \subset \nabla$ as 
the thick interval overlaid in gray if $\dim \nablaf =1$ and as 
the polytope filled in gray if $\nablaf = \nabla$.
Also, the marker ``\textrm{Mfp}'' indicates that $\nablaf \subsetm \nabla$ is a Mori fiber polygon.
If we write the conjugate and the composition of elementary links as in ordinary maps, 
the sequence \pref{eq:cremona} is expressed as
$S\nabla = (U l_{-\infty}^- U^{-1})\circ (U l_0^+ U^{-1})\circ l_0^-\circ l_{-\infty}^+(\nabla)$.
Similarly, as another example, the pair $\nabla_m$ and $S\nabla_m$ is connected by a sequence
$S\nabla_m=(Sl_{m-1}^+S^{-1})\circ \cdots \circ (S l_1^+ S^{-1})\circ l^+ \circ l_1^- \circ \cdots \circ l_{m-1}^-(\nabla_m)$.
Remaining four pairs are left to the reader.
\end{proof}

\begin{proof}[{Non-constructive proof of \pref{th:polygon}}]
   Let us set $\bF_{-\infty} = \bP^2$ for convenience.
   For any sequence $s$ of Sarkisov links in \pref{eq:2dimlinks},
   we write $M(s)$ as the maximal $m$ such that $\bF_m$ 
   appears in the sequence $s$, and $M(\bF_m)=m$ for all $m\in \lc -\infty, 0, 1, \dots \rc$.
   \begin{lemma}
      \label{lm:smallseq}
   Let $p: X\rightarrow S$ and $p':X'\rightarrow S'$ be smooth rational Mori fiber surfaces and 
   $\varphi: X \dashrightarrow X'$ be any birational map.
   Then $\varphi$ is decomposed into a sequence $s$ of Sarkisov links 
   such that $M(s) \le \max\lc 1, M(X), M(X')\rc$. 
   \end{lemma}
   \begin{proof}
   By the Sarkisov program in two dimensions,
   any birational map $\varphi: X\dashrightarrow X'$ can be decomposed into 
   a sequence of Sarkisov links. 
   Suppose $M(s) > \max\lc 1,M(X), M(X')\rc$ for a sequence $s$ decomposing $\varphi$.
   Then $s$ should be locally in the following form:
   \begin{equation}
      \label{eq:maxM}
      s:
      \begin{tikzcd}
      \cdots \arrow[r,dashrightarrow] & 
      \bF_{M(s)-1} \arrow[r,dashrightarrow, "\el_x"] &
       \bF_{M(s)} \arrow[r,dashrightarrow, "\el_y"] & \bF_{M(s)-1} \arrow[r,dashrightarrow] &\cdots
      \end{tikzcd},
   \end{equation}
   where $x\in \bF_{M(s)-1}$ is on the curve of negative self-intersection,
   $y\in \bF_{M(s)}$ is away from the curve of negative self-intersection,
   and $\el_x$ and $\el_y$ are elementary transforms \pref{eq:elem}.
   If $y$ is not an infinitely near point to $x$ (i.e, 
   if $y$ does not lie on the strict transform of the exceptional curve of the blow-up of $\bF_{M(s)-1}$ at $x$),
   one can replace $s$ with the following sequence $s'$ by switching the order of taking elementary transforms:
   \begin{equation}
      \label{eq:maxMfinite}
      s':
      \begin{tikzcd}
      \cdots \arrow[r,dashrightarrow] & 
      \bF_{M(s)-1} \arrow[r,dashrightarrow, "\el_y"] &
       \bF_{M(s)-2} \arrow[r,dashrightarrow, "\el_x"] & \bF_{M(s)-1} \arrow[r,dashrightarrow] &\cdots
      \end{tikzcd},
   \end{equation}
   where we denote the corresponding points $x \in \bF_{M(s)-2}, y\in \bF_{M(s)-1}$
   as the same symbols.
   Similarly, if $y$ is an infinitely near point to $x$ (which is only redundant in toric case),
   by taking another elementary transform at $z\in \bF_{M(s)-1}$ in advance, 
   one can replace $s$ with the following sequence $s'$
   (by using elementary relations in the sense of \cite{MR3123306} twice):
   \begin{equation}
      \label{eq:maxMfinite}
      s':
      \begin{tikzcd}[column sep = 20pt]
      \cdots \arrow[r,dashrightarrow] & 
      \bF_{M(s)-1} \arrow[r,dashrightarrow, "\el_z"] &
       \bF_{M(s)-2} \arrow[r,dashrightarrow, "\el_x"] & 
       \bF_{M(s)-1} \arrow[r,dashrightarrow, "\el_y"] & 
       \bF_{M(s)-2} \arrow[r,dashrightarrow, "\el_z^{-1}"] & \bF_{M(s)-1} \arrow[r,dashrightarrow] &\cdots
      \end{tikzcd},
   \end{equation}
   where $z$ is chosen away from neither the curve of negative 
   self-intersection nor the fiber passing through $x$,
   and $\el_z^{-1}$ denotes the inverse of $\el_z$ 
   away from 
   the fiber passing through $x$.
   In either case, the number of times that $\bF_{M(s)}$ appears is strictly decreased in $s'$.
   Thus, after a finite number of replacing, one may have $M(s')<M(s)$.
   Repeating the same process, 
   one obtains a sequence $s$ that satisfies $M(s) \le \max\lc 1, M(X), M(X')\rc$ in the end.
   \end{proof}

   Let $\nabla$ and $\nabla'$ be any pair of reflexive, and hence, canonical polygons.
   By \pref{pr:twodim},
   one may assume that 
   those are canonical Mori fiber polygons by 
   repeating reductions in advance. 
   By \pref{lm:lift},
   for any pair of canonical Mori fiber polygons $\nablaf\subsetm \nabla$ and $\nablaf'\subsetm \nabla'$,
   there is a pair of smooth toric Mori fiber surfaces $p: X\rightarrow S$ and $p': X'\rightarrow S'$
   which respectively represent $\nablaf\subsetm \nabla$ and $\nablaf'\subsetm \nabla'$
   and the natural birational map $\varphi: X\dashrightarrow X'$ 
   (all of which are uniquely determined in two dimensions).
   Note that $X$ and $X'$ are either $\bF_m$ ($m\le 2$) in this case.
   Clearly from the proof, the sequence in \pref{lm:smallseq} 
   can be chosen as a sequence of toric Sarkisov links for 
   the equivariant birational map $\varphi: X \dashrightarrow X'$.
   Therefore, 
   the sequence of Sarkisov links can be coarse-grained into a sequence of elementary links 
   consisting only of Fano polygons, which
   connects $\nablaf\subsetm \nabla$ and $\nablaf'\subsetm \nabla'$ again by \pref{pr:twodim}.
   From the concrete description of elementary links, 
   all the intermediate Fano polygons turn ont to be canonical.
   The same argument also works for terminal Mori fiber polygons. 
   This completes the proof.
\end{proof}

\begin{remark}
   \pref{lm:smallseq} is a variant of the classical Noether--Castelnuovo theorem, which states that
   the Cremona group $\Bir(\bP^2)$ is generated by automorphisms in $\Aut(\bP^2)\simeq PGL(3,\bC)$
   and the standard quadratic transformation $\sigma: \bP^2 \dashrightarrow \bP^2$ written as 
   $[a:b:c]\mapsto [\frac{1}{a}:\frac{1}{b}:\frac{1}{c}]$.
   In fact, 
   for a given birational automorphism of $\bP^2$,
   one can easily rewrite the sequence obtained in \pref{lm:smallseq}
   into the composition of 
   the transformations described as
   \begin{equation}
      \label{eq:quadratic}
      \begin{tikzcd}
       h^{-1} \circ \sigma \circ  g: \bP^2 \arrow[r,dashrightarrow, "\mathrm{bl}_x^{-1}"] & 
       \bF_1 \arrow[r,dashrightarrow, "\el_y"] & 
       \bF_0 \arrow[r,dashrightarrow, "\el_z"] & \bF_1 \arrow[r, rightarrow] & \bP^2
      \end{tikzcd}
   \end{equation}
   by inserting redundant Sarkisov links.
   Here in \pref{eq:quadratic}, $x,y,z \in \bP^2$ are any three points not on a line,
   $g\in \Aut(\bP^2)$ is an automorphism which moves $x, y, z$ to the torus fixed points, and
   $h\in \Aut(\bP^2)$ is defined to produce $\sigma$
   by fixing torus action depending on the choice of $g$.
\end{remark}

To clarify the relationship between \pref{lm:smallseq} and Reid's fantasy, 
we introduce the notion of geometric transitions for Calabi--Yau elephants in arbitrary dimensions.
\begin{definition}
   Let $X$ and $X'$ be normal projective varieties 
   whose general elephants $D$ and $D'$ are Calabi--Yau varieties,
   and $\varphi: X \rightarrow X'$ be a birational contraction.
   We say that the Calabi--Yau varieties $D$ and $D'$ are related by a 
   \emph{geometric transition associated with $\varphi$}
   if $\varphi|_D: D \rightarrow \varphi_* D$ is a birational contraction
   to a normal variety,
   and $\varphi_*D$ and $D'$ are related by a flat deformation inside $X'$.
   The opposite operation is also called a geometric transition associated with $\varphi^{-1}$.
   Furthermore, 
   if $\varphi|_D$ is an isomorphism, and $\varphi_*D$ and $D'$ are related by a smooth deformation inside $X'$,
   we say that $D$ and $D'$ are related by a 
   \emph{smooth transition}.
\end{definition}
\begin{corollary}
   \label{cr:ellip}
   Let $X$ and $X'$ be smooth projective rational surfaces 
   whose general elephants $E$ and $E'$ are smooth elliptic curves.
   Then any birational map $\varphi: X\dashrightarrow X'$
   is decomposed into a sequence of blow-ups and blow-downs 
   such that  $E$ and $E'$ are connected 
   via a sequence of smooth transitions associated with them.
\end{corollary}
\begin{proof}
   Let $X$ be a smooth projective rational surface whose general elephant $E$ is a smooth elliptic curve.
   Then any $(-1)$-curve $C$ on $X$ intersects with $E$ at one point by the adjunction formula $C.(C+K)=-2$.
   Hence, a contraction of $(-1)$-curves sends $E$ to its isomorphic image
   without changing the property of being an elephant.
   By contracting all $(-1)$-curves in advance,
   one may assume that $X$ and $X'$ are smooth rational Mori fiber surfaces 
   whose elephants are elliptic curves, i.e., $\bF_m$ ($m\le 2$).
   Now, the birational map $\varphi: X\dashrightarrow X'$ is decomposed into a sequence of 
   Sarkisov links obtained by \pref{lm:smallseq}.
   In each Sarkisov link in the sequence, 
   all birational models allow elliptic elephants since they are described by reflexive polygons in
   its toric description. 
   Thus, the elliptic elephants are related via a smooth transition at each step of the Sarkisov links.
   More precisely, for a blow-up at a point, 
   one can smoothly specialize $E$ to another elephant passing through the center of the blow-up
   and send it to its isomorphic strict transform that is again an elephant in a new ambient space.
\end{proof}
\begin{remark}
   \pref{cr:ellip} is compared with the Sarkisov program for Mori fibered 
   Calabi--Yau pairs established in arbitrary dimensions by \cite{MR3504536},
   which shows,
   in the setting of \pref{cr:ellip},
   $E$ and $E'$ are connected only by a sequence of strict transforms if
   $\varphi: (X, E) \dashrightarrow (X',E)$ is a volume-preserving birational map of Calabi--Yau pairs, i.e., 
   $p^*(K_X + E)= {q}^*(K_{X'}+E')$ for a common log resolution $(p,q):W\rightarrow X\times X'$.
\end{remark}
\bibliographystyle{amsalpha}
\bibliography{reflexive}
\end{document}